\documentclass[11pt]{article}
\usepackage[a4paper, tmargin=3cm, bmargin=3.0cm, lmargin=2.0cm, rmargin=2.0cm, textheight=24cm, textwidth=16cm]{geometry}
\usepackage{setspace}
\usepackage{amsfonts}
\usepackage{epsfig}
\usepackage{latexsym}
\usepackage{amsthm}
\usepackage{amsmath}
\usepackage{amssymb}
\usepackage{array}
\usepackage{enumerate}
\usepackage{graphicx}
\usepackage{breqn}

\newtheorem{theo}{Theorem}[section]

\newtheorem{lemma}{Lemma}[section]

\newtheorem{re}{Remark}[section]

\title{Resistance matrices of graphs with matrix weights}

\author{Fouzul Atik\thanks{Theoretical Statistics and Mathematics Unit, Indian Statistical Institute, New Delhi 110 016, India. Email: fouzulatik@gmail.com}  \and  R. B. Bapat\thanks{Theoretical Statistics and Mathematics Unit, Indian Statistical Institute, New Delhi 110 016, India. Email: rbb@isid.ac.in} \and M. Rajesh Kannan\thanks{Department of Mathematics, Indian Institute of Technology Kharagpur, Kharagpur, India. Email: rajeshkannan@maths.iitkgp.ernet.in, rajeshkannan1.m@gmail.com }
}
\date{\today}
\begin{document}
\maketitle
\baselineskip=0.25in

\begin{abstract}
The \emph{resistance matrix} of a simple connected graph $G$ is
denoted by $R$, and is defined by $R =(r_{ij})$, where $r_{ij}$ is
the resistance distance between the vertices $i$ and $j$ of $G$. In
this paper, we consider the resistance matrix of weighted graph with
edge weights being positive definite matrices of same size.  We
derive a formula for the determinant and the inverse of the
resistance matrix. Then, we establish an interlacing inequality for
the eigenvalues of resistance
and Laplacian matrices. Using this interlacing inequality, we obtain the inertia of the resistance matrix.
\end{abstract}
\section{Introduction}\label{section1}
Consider an $n$-vertex connected graph $G=(V,E)$, where
$V=\{1,2,\dots,n\}$ is the vertex set and $E=\{e_1,e_2,\dots,e_m\}$
is the edge set. If two vertices $i$ and $j$ of $G$ are adjacent, we
denote it by $i\sim j$.  A \emph{(positive scalar) weighted graph}
is graph in which every edge is assigned  a weight which is a
positive number. If  weight of the each edge is $1$, the graph is
called an \emph{unweighted} graph. The weight of a path is the sum
of the weights of the edges lies in the path. The \emph{distance
matrix} $D(G)$ of a weighted graph $G$ is an $n\times n$ matrix
$(d_{ij})$, where $d_{ij}$ is the minimum of weights of all paths
from the vertex $i$ to  the vertex $j$. The \emph{Laplacian matrix}
$L$ of a weighted graph $G$ is an $n\times n$ matrix defined as
follows: For $i,j\in [n] = \{1, \dots, n\}$, $i\neq j$, the
$(i,j)^{th}$-element of $L$ is $-\frac{1}{w_{ij}}$ if the vertices
$i$ and $j$ are adjacent, where $w_{ij}$ denotes the weight of the
edge joining the vertices $i$ and $j$, and  it is zero, otherwise.
For $i\in [n]$, the $(i,i)^{th}$-entry of  the matrix $L$ is $\sum
_{j \sim i}\frac{1}{w_{ij}}$. Let $e_{ij}$ denote the $n\times 1$
vector with $1$ at the $i^{th}$ place, $-1$ at the $j^{th}$ place,
and zero elsewhere. A matrix $H$ is called a \emph{generalized
inverse}, or $g$-\emph{inverse} of $L$ if $LHL = L.$ It is known
\cite{bapatbook} that for a graph with positive edge weights, the
scalar $e_{ij}^{\prime}He_{ij}$ is same for any $g$-inverse $H$ of
$L$ where $e_{ij}^{\prime}$ is the transpose of $e_{ij}$. The
\emph{resistance distance} between the vertices $i$ and $j$, denoted
by $r(i,j)$, defined as
$r(i,j)=e_{ij}^{\prime}He_{ij}=h_{ii}+h_{jj}-h_{ij}-h_{ji}$, where
$H$ is a $g$-inverse of $L$. In particular, if
$L^\dagger=(k_{ij})$ is the Moore-Penrose inverse(defined
in Section \ref{someresults}) of $L$, then
$r(i,j)=k_{ii}+k_{jj}-2k_{ij}$. The
\emph{resistance matrix} $R$ of a graph is an $n\times n$ matrix
defined as follows: For $i,j\in [n]$, the $(i,j)^{th}$-entry of $R$
is defined as $r_{ij}=r(i,j)$. For a tree, the resistance distance
coincides with the classical distance.

The distance matrix of a connected graph, in particular of a tree,
has been studied by researchers for many years. One of the
remarkable result about the distance matrix of a tree states that,
if $T$ is a tree on $n$ vertices, then the determinant of the
distance matrix $D$ of $T$ is $(-1)^{n-1}(n-1)2^{n-2}$, which is
independent of the structure of the underlying tree \cite{grah2}. In
\cite{grah1}, a formula for the inverse of the distance matrix of a
tree is obtained. An extension of these results for the weighted
trees, where the weights are positive scalars, were considered in
\cite{bapat1}. In \cite{bapat2}, formulae for  determinant  and
inertia of the distance matrix of a weighted tree with edge weights
being square matrices of same order. Resistance distance matrices
are one of the important classes of matrices, associated with
graphs, considered in the chemical literature. The notion of Wiener
index based on the resistance distance has proposed in \cite{klein}.
The formulae for the determinant and the inverse of the resistance
matrix of scalar weighted graph are obtained in \cite{RB1}. Some
inequalities for the eigenvalues of the distance matrix of a tree
and the eigenvalues of the resistance matrix of any connected graph
are established in \cite{merr} and \cite{sun}, respectively.  We
refer to \cite{kool,liu,xiao,xiao2,zhou2} for more information on
the resistance distance.

In this paper, we consider the resistance matrix of any weighted
graph whose edge weights are positive definite matrices of same
size. We derive a formula for the determinant(Theorem
\ref{rdeterminant}) and the inverse(Theorem \ref{rinverse}) of the
resistance matrix. We establish an interlacing inequality for the
eigenvalues of resistance and Laplacian matrices(Theorem
\ref{Rinterlace}). Using this, we obtain the inertia of the
resistance matrix(Theorem \ref{Rinertia}).

This article is organized as follows: In section \ref{someresults}, we first recall some of the definitions and results, which will be used in the subsequent part of the paper.  Then we establish two important properties, one of them related to the principal submatrix of Moore-Penrose inverse of a positive semidefinite matrix(Lemma \ref{nonsingprin}) and another is about the cofactor of block matrices(Lemma \ref{equalcofactor}). In section \ref{main}, we derive some of the properties of the resistance matrix with positive definite matrix edge weights. Using these properties, in section \ref{det-inv-iner}, we obtain formulae for the determinant, inverse and inertia of the resistance matrix with positive definite matrix edge weights.

\section{Some useful results}\label{someresults}
Let $B$ be an $n \times n$ matrix and let $S,K \subseteq [n]$. We
denote by $B[S,K]$, the matrix obtained by selecting the rows of $B$
indexed by $S$, and the columns of $B$ indexed by $K$.  On the other
hand, $B(S,K)$ is the matrix obtained by deleting the rows of $B$
indexed by $S$, and the columns of $B$ indexed by $K$. Similarly, if
$x$ is an $n\times 1$ vector, then $x[S]$ will denote the subvector
of $x$ indexed by indices in $S$. The Kronecker product of matrices $A=(a_{ij})$ of size
$m\times n$ and $B$ of size $p\times q$, denoted by $A\otimes B$, is
defined to be the $mp\times nq$ block matrix $(a_{ij}B)$. It is
known that for matrices $M$, $N$, $P$ and $Q$ of suitable sizes,
$MN\otimes PQ=(M\otimes P)(N\otimes Q)$\cite{hor}.

Let $A$ be an $n \times n$ matrix partitioned as \begin{eqnarray}\label{partmatrix}A = \left[
                     \begin{array}{cc}
                       A_{1,1} & A_{1,2} \\
                       A_{2,1} & A_{2,2}
                     \end{array}
                   \right]
\end{eqnarray} such that the diagonal blocks $A_{1,1}$ and $A_{2,2}$ are square matrices and the matrix $A_{1,1}$ is invertible. Then the Schur complement of $A_{1, 1}$ in $A$ is defined to be the matrix $A_{2,2} - A_{2,1}A_{1,1}^{-1}A_{1,2}$. The determinant of the matrix $A$ and determinant of the Schur complement of $A_{1,1}$ in $A$ has the following interesting relationship: $detA = det(A_{1, 1})(det A_{2,2} - A_{2,1}A_{1,1}^{-1}A_{1,2}).$ This is called  the Schur  formula for the determinant.

Let $B^{'}$ denote the transpose of the matrix $B$. Let $A$ be an $m
\times n$ matrix. A matrix $G$ of order $n\times m$ is said to be a
generalized inverse (or a $g$-inverse) of $A$ if $AGA=A$. A matrix
$G$ of order $n \times m$ is said to be the \emph{Moore-Penrose
inverse} of $A$ if it satisfies (i) $AGA=A$, (ii) $GAG=G$, (iii)
$(AG)'=AG$ and (iv) $(GA)'=GA$. We denote the Moore-Penrose inverse
of the matrix $A$ by $A^\dagger$.

In the next lemma, we prove that if a principal submatrix $A[S,S]$ of a positive semidefinite matrix $A$ is invertible, then the principal submatrix of $A^\dagger$ corresponds to the same index set $S$ is also invertible. i.e.,  $A^\dagger[S,S]$ is invertible.

\begin{lemma}\label{nonsingprin} Let $A$ be an
$n \times n$ positive semidefinite matrix, whose rows and columns
are indexed by $[n]$. For $S\subseteq[n]$, if $A[S,S]$ is
nonsingular then $A^\dagger[S,S]$ is also nonsingular.
\end{lemma}
\begin{proof}Let $A$ be of rank $r(\leq n)$. For
$i=1,2,\cdots,r$, let $\lambda_i$ be the nonzero eigenvalues of $A$ with
corresponding eigenvectors $x_i$. Then the spectral decomposition of
$A$ is given by
\begin{eqnarray}
\label{specdecom} A=\lambda_1 x_1 x_1'+\lambda_2 x_2
x_2'+\cdots+\lambda_r x_r x_r'.
\end{eqnarray}
Let $S\subseteq[n]$ with $|S|=p~(\leq r)$ and $A[S,S]$ is
nonsingular. Then
\begin{eqnarray*}
A[S,S]&=&\lambda_1 x_1[S] x_1'[S]+\lambda_2 x_2[S]
x_2'[S]+\cdots+\lambda_r x_r[S] x_r'[S]\\
&=&\left[\sqrt{\lambda_1} x_1[S]~\sqrt{\lambda_2 }x_2[S]
~\cdots~\sqrt{\lambda_r} x_r[S]\right] \left[
                                                       \begin{array}{c}
                                                         \sqrt{\lambda_1} x_1'[S] \\
                                                         \sqrt{\lambda_2} x_2'[S] \\
                                                         \cdots \\
                                                         \sqrt{\lambda_r} x_r'[S] \\
                                                       \end{array}
                                                     \right].
\end{eqnarray*}
Since the rank of the matrix $A[S,S]$ is $p$, the rank of $\left[\sqrt{\lambda_1}
x_1[S]~\sqrt{\lambda_2 }x_2[S] ~\cdots~\sqrt{\lambda_r}
x_r[S]\right]$ is also $p$. Again
\begin{eqnarray*}
\left[\sqrt{\lambda_1} x_1[S]~\sqrt{\lambda_2 }x_2[S]
~\cdots~\sqrt{\lambda_r} x_r[S]\right]=\left[x_1[S]~x_2[S] ~\cdots~
x_r[S]\right]\left[
               \begin{array}{cccc}
                 \sqrt{\lambda_1} & 0 & \cdots & 0 \\
                 0 & \sqrt{\lambda_2} & \cdots & 0 \\
                 \cdots & \cdots & \ddots & \cdots \\
                 0 & 0 & \cdots & \sqrt{\lambda_r} \\
               \end{array}
             \right].
\end{eqnarray*}
Thus we get rank of $\left[x_1[S]~x_2[S] ~\cdots~ x_r[S]\right]$ is
$p$. Then rank of $$\left[x_1[S]~x_2[S] ~\cdots~ x_r[S]\right]\left[
               \begin{array}{cccc}
                 \frac{1}{\sqrt{\lambda_1}} & 0 & \cdots & 0 \\
                 0 & \frac{1}{\sqrt{\lambda_2}} & \cdots & 0 \\
                 \cdots & \cdots & \ddots & \cdots \\
                 0 & 0 & \cdots & \frac{1}{\sqrt{\lambda_r}} \\
               \end{array}
             \right]=\left[\frac{1}{\sqrt{\lambda_1}} x_1[S]~\frac{1}{\sqrt{\lambda_2}} x_2[S] ~\cdots~\frac{1}{\sqrt{\lambda_r}}
x_r[S]\right]$$ is also $p$. We note that the spectral decomposition of
$A^+$ is $$A^\dagger=\frac{1}{\lambda_1} x_1 x_1'+\frac{1}{\lambda_2} x_2
x_2'+\cdots+\frac{1}{\lambda_r} x_r x_r'.$$ Thus
$$A^\dagger[S,S]=\frac{1}{\lambda_1} x_1[S,S] x_1'[S,S]+\frac{1}{\lambda_2}
x_2[S,S] x_2'[S,S]+\cdots+\frac{1}{\lambda_r} x_r[S,S] x_r'[S,S]$$ is
of rank $p$ and hence $A^\dagger[S,S]$ is also nonsingular.
\end{proof}
Next we define the cofactor of a submatrix. Let
$S=\{i_1,i_2,\cdots,i_k\}$ and $K=\{j_1,j_2,\cdots,j_k\}$ be subsets
of $[n]$. For a matrix $A$ of order $n\times n$, cofactor of
the submatrix $A[S,K]$ in $A$ is
$$A_{SK}=(-1)^{i_1+i_2+\cdots+i_k+j_1+j_2+\cdots+j_k}\det A(S,K).$$
Consider a square matrix $A$ whose rows and columns are indexed by
elements in $X=\{1,2,\cdots,ns\}$ and a partition $
\pi=\{X_1,X_2,\cdots,X_n\}$ of $X$ where $|X_i|=s$ for all
$i=1,2,\cdots,n$. We partition the matrix $A$ according to $\pi$ as
\begin{eqnarray}\label{Apartition}
A=\left[
  \begin{array}{cccc}
A_{1,1} & A_{1,2} & \cdots & A_{1,n}\\
A_{2,1} & A_{2,2} & \cdots & A_{2,n}\\
\cdots & \cdots &\cdots &\cdots\\
A_{n,1} & A_{n,2} & \cdots & A_{n,n}\\
\end{array} \right],
\end{eqnarray}
where each $A_{i,j}=A[X_i,X_j]$ is a submatrix (block) of $A$ whose
rows and columns are indexed by elements of $X_i$ and $X_j$,
respectively.
\begin{lemma}\label{equalcofactor} Let $A$ be an
$ns \times ns$ square matrix and it is partitioned as in (\ref{Apartition}). Let $\sum_{j=1}^nA_{i,j}=0$ and  $\sum_{i=1}^nA_{i,j}=0$ for all $i,j\in\{1,2,\cdots,n\}$. Then
the cofactor of any two blocks $A_{i,j}$ and $A_{k,l}$ in $A$ are equal.
\end{lemma}
\begin{proof}We have
\begin{eqnarray*}
\mbox{Cofactor
of}~A_{1,1}&=&(-1)^{1+2+\cdots+s+1+2+\cdots+s}\det\left[
\begin{array}{cccc}
A_{2,2} & A_{2,3} & \cdots & A_{2,n}\\
A_{3,2} & A_{3,3} & \cdots & A_{3,n}\\
\cdots & \cdots &\cdots &\cdots\\
A_{n,2} & A_{n,3} & \cdots & A_{n,n}\\
\end{array} \right]\\
&=&\det \left[
\begin{array}{cccc}
\sum_{j=2}^nA_{2,j} & A_{2,3} & \cdots & A_{2,n}\\
\sum_{j=2}^nA_{3,j} & A_{3,3} & \cdots & A_{3,n}\\
\cdots & \cdots &\cdots &\cdots\\
\sum_{j=2}^nA_{n,j} & A_{n,3} & \cdots & A_{n,n}\\
\end{array} \right]\\
&=&\det\left[
\begin{array}{cccc}
-A_{2,1} & A_{2,3} & \cdots & A_{2,n}\\
-A_{3,1} & A_{3,3} & \cdots & A_{3,n}\\
\cdots & \cdots &\cdots &\cdots\\
-A_{n,1} & A_{n,3} & \cdots & A_{n,n}\\
\end{array} \right]~~[\mbox{As $\sum_{j=1}^nA_{i,j}=0$}]\\
&=&(-1)^s\det\left[
\begin{array}{cccc}
A_{2,1} & A_{2,3} & \cdots & A_{2,n}\\
A_{3,1} & A_{3,3} & \cdots & A_{3,n}\\
\cdots & \cdots &\cdots &\cdots\\
A_{n,1} & A_{n,3} & \cdots & A_{n,n}\\
\end{array} \right]\\
&=&(-1)^{(1+2+\cdots+s)+(s+1+s+2+\cdots+s+s)}\det\left[
\begin{array}{cccc}
A_{2,1} & A_{2,3} & \cdots & A_{2,n}\\
A_{3,1} & A_{3,3} & \cdots & A_{3,n}\\
\cdots & \cdots &\cdots &\cdots\\
A_{n,1} & A_{n,3} & \cdots & A_{n,n}\\
\end{array} \right]\\
&=&\mbox{Cofactor of}~A_{1,2}.
\end{eqnarray*}
Similarly, we can prove that cofactor of $A_{i,j}$ is equal to
cofactor of $A_{i,k}$, for any $i,j,k$. By using the assumption
$\sum_{i=1}^nA_{i,j}=0$ and applying the above steps to the rows of $A$, we
get the cofactor of $A_{i,j}$  equals to cofactor of $A_{k,j}$, for any
$i,j,k$.
\end{proof}
\section{Resistance matrix with matrix weights}\label{main}
Let $G$ be a connected weighted graph with $n$ vertices, $m$ edges
and all the edge weights of $G$ are positive definite matrices of
order $s\times s$. For a positive definite matrix $A$, let
$\sqrt{A}$ denote the unique positive definite square root of $A$.
The Laplacian matrix $L$ of the graph $G$ is the $ns\times ns$ block
matrix whose row and column blocks are index by the vertices of $G$
and $(i,j)^{th}$ block of $L$   equals to the sum of the inverse of
weight matrices of the edges incident with the vertex $i$ if $i=j$,
and is equals to $-W^{-1}$ if there is an edge between the vertices
$i$ and $j$, where $W$ is the weight of the edge, and zero matrix
otherwise. We assign an orientation to each edge of $G$. Then the
vertex edge incidence matrix $Q$ of $G$ is the $ns\times ms$ block
matrix whose row and column blocks are index by the vertices and the
edges of $G$, respectively. The $(i,j)^{th}$ block of $Q$ is zero
matrix if the $i^{th}$ vertex and the $j^{th}$ edge are not incident
and it is $\sqrt{W_j}^{-1}$ (respectively, $-\sqrt{W_j}^{-1}$) if
the $i^{th}$ vertex and the $j^{th}$ edge are incident, and the edge
originates (respectively, terminates) at the $i^{th}$ vertex, where
$W_j$ is the weight of the $j^{th}$ edge. In this case, it is easy
to see that, the Laplacian matrix $L$ can be written as
$L=QQ^{'}$.

Let $L^\dagger$ be the Moore-Penrose inverse of the Laplacian matrix
$L$ of a graph $G$, and let $K _{i,j},~i,j=1,2,\cdots,n$, be the $s\times s$
blocks of $L^\dagger$. For the graph $G$, we define the resistance
matrix $R$ by an $n\times n$ block matrix whose $(i,j)^{th}$ block
$R_{i,j}$ is defined by
\begin{eqnarray}
\label{rdef1}R_{i,j}=K _{i,i}+K _{j,j}-2K _{i,j},~i,j=1,2,\cdots,n.
\end{eqnarray}
Since the Moore-Penrose inverse of a matrix always exists and it is unique, so the
definition of resistance matrix is well defined. Also,  if
$s=1$, then the above definition of resistance matrix coincides with
the definition of resistance matrix of positive scalar weighted
graphs.

In the next theorem, we prove the matrix $L+\frac{1}{n}J\otimes I_s$ is nonsingular.

\begin{theo}
Let $G$ be a connected graph on $n$ vertices such that the edge weights
are $s\times s$ positive definite matrices.  Then  $L+\frac{1}{n}J\otimes I_s$ is nonsingular.
\end{theo}
\begin{proof} Since each weight matrix is positive definite,
we get $L$ is  positive semidefinite and rank of $L$ is $ns-s$
\emph{\cite{fou-raj-bap}}. Let the eigenvalues of $L$ be
$\lambda_1,\lambda_2,\cdots,\lambda_{ns-s},0,0,\cdots,0,$ where
multiplicity of $0$ is $s$. Now according to the construction of
$L$, we have $L(\emph{\textbf{1}}\otimes I_s)=0.$ Thus each of the
column vector of $\emph{\textbf{1}}\otimes I_s$, is an eigenvector
of $L$ corresponding to the eigenvalue $0$. Note that column vectors
of $\emph{\textbf{1}}\otimes I_s$ are linearly independent. Let
$x_i$ be an eigenvector of $L$ corresponding to the eigenvalue
$\lambda_i\neq 0$ for $i=1,2,\cdots,ns-s$. As $L$ is symmetric,
\begin{eqnarray*}&&x_i'(\emph{\textbf{1}}\otimes
I_s)=0\\
&\Rightarrow& (\emph{\textbf{1}}'\otimes I_s)x_i=0\\
&\Rightarrow& (\emph{\textbf{1}}\emph{\textbf{1}}'\otimes I_s)x_i=0\\
&\Rightarrow& (J\otimes I_s)x_i=0.
\end{eqnarray*}
Then
\begin{eqnarray*}(L+\frac{1}{n}J\otimes I_s)(\emph{\textbf{1}}\otimes I_s)&=&L(\emph{\textbf{1}}\otimes I_s)+\frac{1}{n}(J\otimes I_s)(\emph{\textbf{1}}\otimes I_s)\\
&=&\frac{1}{n}(J\emph{\textbf{1}}\otimes I_s)\\
&=&(\emph{\textbf{1}}\otimes I_s).
\end{eqnarray*}
Thus each of the column vector of
$\emph{\textbf{1}}\otimes I_s$, is an eigenvector of
$L+\frac{1}{n}J\otimes I_s$ corresponding to the eigenvalue $1$.
Now, for each $i=1,2,\cdots,ns-s$,
\begin{eqnarray*}(L+\frac{1}{n}J\otimes I_s)x_i&=&L x_i+\frac{1}{n}(J\otimes I_s)x_i\\
&=&L x_i\\
&=&\lambda_i x_i.
\end{eqnarray*}
Thus $x_i$ is also an eigenvector of $L+\frac{1}{n}J\otimes I_s$
corresponding to the eigenvalue $\lambda_i$ for $i=1,2,\cdots,ns-s$.
Hence the eigenvalues of $L+\frac{1}{n}J\otimes I_s$ are
$\lambda_1,\lambda_2,\cdots,\lambda_{ns-s},1,1,\cdots,1,$ where
multiplicity of $1$ is $s$. Thus $L+\frac{1}{n}J\otimes I_s$ is invertible.
\end{proof}
In the previous theorem, we have seen that $L+\frac{1}{n}J\otimes I_s$
is nonsingular. Now set $X=(L+\frac{1}{n}J\otimes I_s)^{-1}$.
Since $L(J\otimes I_s)=(J\otimes I_s)L=\textbf{0}$, we have
\begin{eqnarray}\nonumber&&(L+\frac{1}{n}J\otimes I_s)L=L(L+\frac{1}{n}J\otimes I_s)\\
\nonumber&\Rightarrow& X(L+\frac{1}{n}J\otimes I_s)LX=XL(L+\frac{1}{n}J\otimes I_s)X\\
\label{xlelx}&\Rightarrow& LX=XL.
\end{eqnarray}
Using this observation, in the next theorem, we derive a formula for the Moore-Penrose inverse of the Laplacian matrix associated with the graph.

\begin{theo}\label{lplus}
Let $G$ be a connected graph on $n$ vertices such that the edge weights
are $s\times s$ positive definite matrices and let $L$ be the Laplacian matrix associated with $G$. Then  $L^\dagger =X-\frac{1}{n}J\otimes I_s$.
\end{theo}
\begin{proof}
\begin{eqnarray}\nonumber L(X-\frac{1}{n}J\otimes I_s)L&=&L X L\\
\nonumber &=&L[I_{ns}-\frac{1}{n}X(J\otimes I_s)] \\
\nonumber &=&L-\frac{1}{n}LX(J\otimes I_s)\\
\nonumber &=&L-\frac{1}{n}XL(J\otimes I_s) ~~~\mbox{[By (\ref{xlelx})]}\\
\label{l1} &=&L.
\end{eqnarray}
Again applying (\ref{xlelx}), we get
\begin{eqnarray}\nonumber&&X(J\otimes I_s)=(J\otimes I_s)X\\
\nonumber&\Rightarrow& X(J\otimes I_s)(L+\frac{1}{n}J\otimes I_s)=(J\otimes I_s)X(L+\frac{1}{n}J\otimes I_s)\\
\nonumber&\Rightarrow& X[\frac{1}{n}(J\otimes I_s)(J\otimes I_s)]=(J\otimes I_s)\\
\label{xjj}&\Rightarrow& X(J\otimes I_s)=(J\otimes I_s).
\end{eqnarray}
Similarly, we get
\begin{eqnarray}
\label{jxj} (J\otimes I_s)X=(J\otimes I_s).
\end{eqnarray}
Now,
\begin{eqnarray}\nonumber (X-\frac{1}{n}J\otimes I_s)L(X-\frac{1}{n}J\otimes I_s)&=& X L X\\
\nonumber &=&(I_{ns}-\frac{1}{n}XJ\otimes I_s)X\\
\nonumber &=&X-\frac{1}{n}X(J\otimes I_s)X\\
\label{l2} &=&X-\frac{1}{n}J\otimes I_s ~~~\mbox{[By (\ref{xjj}) and
(\ref{jxj})]}.
\end{eqnarray}
Again,
\begin{eqnarray}\nonumber [(X-\frac{1}{n}J\otimes I_s)L]'&=&L'(X-\frac{1}{n}J\otimes I_s)'\\
\nonumber &=&L(X-\frac{1}{n}J\otimes I_s)\\
\nonumber &=&L X\\
\nonumber &=&X L\\
\label{l3} &=&(X-\frac{1}{n}J\otimes I_s)L
\end{eqnarray}
Similarly,
\begin{eqnarray}\label{l4}[L(X-\frac{1}{n}J\otimes I_s)]'=L(X-\frac{1}{n}J\otimes I_s)
\end{eqnarray}
By combining (\ref{l1}), (\ref{l2}), (\ref{l3}) and (\ref{l4}), we
get the result.
\end{proof}
Now, using Theorem \ref{lplus} and equation (\ref{rdef1}), for
$i,j=1,2,\cdots,n$, we have
\begin{eqnarray}\label{rdef2}R_{i,j}=X_{i,i}+X_{j,j}-2X_{i,j}
\end{eqnarray}
Let $\bar{X}$ denote the diagonal block matrix whose diagonal blocks are
$X_{1,1},X_{2,2},\cdots,X_{n,n}$. Then,  by using the previous equation
we get,
\begin{eqnarray}\label{rmatdef}R=\bar{X}(J\otimes I_s)+(J\otimes I_s)\bar{X}-2X
\end{eqnarray}
For $i,j=1,2,\cdots,n$, we define the $s\times s$ matrices $\tau_i$
and the $ns\times s$ matrix $\tau$ as follows:
\begin{eqnarray*}\tau_i=2I_s-\sum_{j\sim
i}W_{i,j}^{-1}R_{j,i},\\
\tau=[\tau_1,\tau_2,\cdots,\tau_n]'
\end{eqnarray*}

\begin{theo}\label{taudef}
Let $G$ be a connected graph on $n$ vertices such that the edge weights
are $s\times s$ positive definite matrices and $L$ be the Laplacian matrix associated with $G$.
If the matrices $\bar{X}$ and $\tau$ are defined as above, then $L\bar{X}(\textbf{1}\otimes I_s)+\frac{2}{n}(\textbf{1}\otimes
I_s)=\tau.$
\end{theo}
\begin{proof}
We have
\begin{eqnarray*}(L+\frac{1}{n}J\otimes I_s)X=I_{ns}
\end{eqnarray*}
Equating the $(i,i)^{th}$ block in both sides, we get
\begin{eqnarray}\nonumber &&\sum_{j\sim i}W_{i,j}^{-1}X_{i,i}-\sum_{j\sim
i}W_{i,j}^{-1}X_{j,i}+\frac{1}{n}\sum_{j=1}^n I_s X_{j,i}=I_s\\
\label{taudef_eq1} &\Rightarrow& \left(\sum_{j\sim
i}W_{i,j}^{-1}\right)X_{i,i}-\sum_{j\sim
i}W_{i,j}^{-1}X_{j,i}+\frac{1}{n}\sum_{j=1}^n X_{j,i}=I_s.
\end{eqnarray}
Now,
\begin{eqnarray*} &&(L+\frac{1}{n}J\otimes I_s)(\emph{\textbf{1}}\otimes I_s)=\emph{\textbf{1}}\otimes I_s\\
&\Rightarrow& X(L+\frac{1}{n}J\otimes
I_s)(\emph{\textbf{1}}\otimes I_s)=X(\emph{\textbf{1}}\otimes I_s)\\
&\Rightarrow& (\emph{\textbf{1}}\otimes I_s)=X(\emph{\textbf{1}}\otimes I_s)\\
&\Rightarrow& (\emph{\textbf{1}}'\otimes I_s)=(\emph{\textbf{1}}'\otimes I_s)X\\
&\Rightarrow& \sum_{j=1}^n X_{j,i}=I_s ~~\mbox{for $i=1,2,\cdots,n$}
\end{eqnarray*}
From (\ref{taudef_eq1}),  we get
\begin{eqnarray}
\label{taudef_eq2} \left(\sum_{j\sim
i}W_{i,j}^{-1}\right)X_{i,i}-\sum_{j\sim
i}W_{i,j}^{-1}X_{j,i}=(1-\frac{1}{n})I_s.
\end{eqnarray}
Now, for $i=1,2,\cdots,n,$
\begin{eqnarray*}\tau_i &=&2I_s-\sum_{j\sim
i}W_{i,j}^{-1}R_{j,i}\\
&=&2I_s-\sum_{j\sim
i}W_{i,j}^{-1}(X_{i,i}+X_{j,j}-2X_{j,i})\\
&=&(1+\frac{1}{n})I_s-\sum_{j\sim i}W_{i,j}^{-1}X_{j,j}+\sum_{j\sim
i}W_{i,j}^{-1}X_{j,i}.~~\mbox{[By (\ref{taudef_eq2})]}
\end{eqnarray*}
Let $\gamma_i$ be the $i^{th}$ block of
$L\bar{X}(\textbf{1}\otimes I_s)+\frac{2}{n}(\textbf{1}\otimes
I_s)$. Then, for $i=1,2,\cdots,n$,
\begin{eqnarray*}\gamma_i &=&\sum_{k=1}^n(L\bar{X})_{i,k}I_s+\frac{2}{n}I_s\\
&=&\sum_{k=1}^n\left(\sum_{j=1}^nL_{i,j}\bar{X}_{j,k}\right)+\frac{2}{n}I_s\\
&=&\sum_{k=1}^nL_{i,k}\bar{X}_{k,k}+\frac{2}{n}I_s\\
&=&L_{i,i}\bar{X}_{i,i}+\sum_{k\sim i} L_{i,k}\bar{X}_{k,k}+\frac{2}{n}I_s\\
&=&\left(\sum_{j\sim i} W_{i,j}^{-1}\right)\bar{X}_{i,i}-\sum_{j\sim i} W_{i,j}^{-1}\bar{X}_{j,j}+\frac{2}{n}I_s\\
&=&(1+\frac{1}{n})I_s-\sum_{j\sim i}W_{i,j}^{-1}X_{j,j}+\sum_{j\sim
i}W_{i,j}^{-1}X_{j,i}~~\mbox{[By (\ref{taudef_eq2})]}\\
&=&\tau_i.
\end{eqnarray*}
\end{proof}
\begin{re}\label{1tau1}
\emph{Using the previous theorem, we get the following identities:
\begin{eqnarray*}(\emph{\textbf{1}}'\otimes I_s)\tau &=&(\emph{\textbf{1}}'\otimes I_s)[L\bar{X}(\emph{\textbf{1}}\otimes
I_s)+\frac{2}{n}(\emph{\textbf{1}}\otimes I_s)]\\
&=&\frac{2}{n}(\emph{\textbf{1}}'\otimes
I_s)(\emph{\textbf{1}}\otimes I_s)\\
&=&2I_s\\
\Rightarrow\tau'(\emph{\textbf{1}}\otimes I_s)&=&2I_s.
\end{eqnarray*}}
\end{re}
\begin{theo}\label{rwiden}
Let $G$ be a connected graph on $n$ vertices such that the edge
weights are $s\times s$ positive definite matrices and $R$ be the
resistance matrix of  $G$. Then
\begin{eqnarray*}\sum_{i=1}^n\sum_{j\sim i}
W_{i,j}^{-1}R_{j,i}=2(n-1)I_s.\end{eqnarray*}
\end{theo}
\begin{proof}
From Theorem \ref{lplus}, we have $L^\dagger  = (X-\frac{1}{n}J\otimes I_s)$. Now,
\begin{eqnarray*}LL^\dagger =L(X-\frac{1}{n}J\otimes I_s)=LX=I_{ns}-\frac{1}{n}(J\otimes I_s)X=I_{ns}-\frac{1}{n}J\otimes I_s.
\end{eqnarray*}
From equation (\ref{rmatdef}), we have
\begin{eqnarray}\nonumber R&=&\bar{X}(J\otimes I_s)+(J\otimes I_s)\bar{X}-2X. \end{eqnarray}
Thus \begin{eqnarray}LR&=&L\bar{X}(J\otimes I_s)-2LX\\
\label{rwideq_1}&=&L\bar{X}(J\otimes
I_s)-2I_{ns}+\frac{2}{n}(J\otimes I_s).
\end{eqnarray}
Equating the $(i,i)^{th}$ block in both sides of (\ref{rwideq_1}) we
get
\begin{eqnarray*}\sum_{j=1}^nL_{i,j}R_{j,i}&=&\sum_{j=1}^n(L\bar{X})_{i,j}(J\otimes
I_s)_{j,i}-2I_s+\frac{2}{n}I_s\\
\Rightarrow L_{i,i}R_{i,i}-\sum_{j\sim i}
W_{i,j}^{-1}R_{j,i}&=&\sum_{j=1}^n(L\bar{X})_{i,j}I_s-(2-\frac{2}{n})I_s\\
\Rightarrow -\sum_{j\sim i}W_{i,j}^{-1}R_{j,i}&=&\sum_{j=1}^n\sum_{k=1}^nL_{i,k}\bar{X}_{k,j}-(2-\frac{2}{n})I_s\\
&=&\sum_{j=1}^nL_{i,j}\bar{X}_{j,j}-(2-\frac{2}{n})I_s.
\end{eqnarray*}
Now,
 \begin{eqnarray*}
 \sum_{i=1}^n\sum_{j\sim i}W_{i,j}^{-1}R_{j,i}&=&-\sum_{i=1}^n\sum_{j=1}^nL_{i,j}\bar{X}_{j,j}+\sum_{i=1}^n(2-\frac{2}{n})I_s\\
&=&-\sum_{j=1}^n\bar{X}_{j,j}\sum_{i=1}^nL_{i,j}+2(n-1)I_s.\\
&=&2(n-1)I_s.
\end{eqnarray*}
\end{proof}

Using Theorem \ref{lplus}, we have the following  identity
\begin{eqnarray}
\nonumber  LRL&=&L(\bar{X}(J\otimes I_s)+(J\otimes
I_s)\bar{X}-2X)L\\
\nonumber  &=&-2LXL\\
\nonumber  &=&-2L(L^\dagger +\frac{1}{n}(J\otimes I_s))L\\
\label{LRL}&=&-2LL^\dagger L=-2L.
\end{eqnarray}
In the next theorem, we establish the matrix $\tau'R\tau$ is positive definite and derive a formula for it.
\begin{theo}\label{taurtaunonsing}
Let $G$ be a connected graph on $n$ vertices such that the edge weights
are $s\times s$ positive definite matrices, and let $L$ and $R$ be the Laplacian matrix  and the resistance matrix of $G$, respectively. If the matrices $\bar{X}$ and $\tau$ are defined as in Theorem \ref{taudef}, then the
matrix $\tau'R\tau$ is  positive definite and is equal to
$2\bar{x}'L\bar{x}+\frac{8}{n}\left[\sum_{i=1}^nX_{ii}-I_s\right]$,
where $\bar{x}=\bar{X}(\emph{\textbf{1}}\otimes I_s)$.
\end{theo}
\begin{proof}We have,
\begin{eqnarray}\nonumber \tau'R\tau&=&[(\textbf{1}'\otimes I_s)\bar{X}L+\frac{2}{n}(\textbf{1}'\otimes I_s)]
                             R[L\bar{X}(\textbf{1}\otimes I_s)+\frac{2}{n}(\textbf{1}\otimes
                             I_s)]\\
\nonumber                     &=&(\textbf{1}'\otimes
I_s)\bar{X}LRL\bar{X}(\textbf{1}\otimes I_s)
                             +\frac{2}{n}(\textbf{1}'\otimes I_s)RL\bar{X}(\textbf{1}\otimes
                             I_s)\\
\label{taurtau}              &&+\frac{2}{n}(\textbf{1}'\otimes
I_s)\bar{X}LR(\textbf{1}\otimes I_s)
                             +\frac{4}{n^2}(\textbf{1}'\otimes I_s)R(\textbf{1}\otimes I_s).
\end{eqnarray}
Now,
\begin{eqnarray*}
(\textbf{1}'\otimes I_s)\bar{X}LRL\bar{X}(\textbf{1}\otimes
I_s)&=&(\textbf{1}'\otimes
I_s)\bar{X}(-2L)\bar{X}(\textbf{1}\otimes I_s)\\
&=&-2\bar{x}'L\bar{x},
\end{eqnarray*}
and
\begin{eqnarray*}
\frac{2}{n}(\textbf{1}'\otimes I_s)\bar{X}LR(\textbf{1}\otimes
I_s)&=&\frac{2}{n}(\textbf{1}'\otimes I_s)\bar{X}(L\bar{X}(J\otimes
I_s)-2I_{ns}+\frac{2}{n}(J\otimes I_s))(\textbf{1}\otimes I_s)~\mbox{[By (\ref{rwideq_1})]}\\
&=&2\bar{x}'L\bar{x}.
\end{eqnarray*}
Similarly, we can derive the following
\begin{eqnarray*}
\frac{2}{n}(\textbf{1}'\otimes I_s)RL\bar{X}(\textbf{1}\otimes
I_s)&=&2\bar{x}'L\bar{x}.
\end{eqnarray*}
Finally,
\begin{eqnarray*}
\frac{4}{n^2}(\textbf{1}'\otimes I_s)R(\textbf{1}\otimes
I_s)&=&\frac{4}{n^2}(\textbf{1}'\otimes I_s)(\bar{X}(J\otimes
I_s)+(J\otimes I_s)\bar{X}-2X)(\textbf{1}\otimes I_s)\\
&=&\frac{4}{n^2}[n(\textbf{1}'\otimes I_s)\bar{X}(\textbf{1}\otimes
I_s)+ n(\textbf{1}'\otimes I_s)\bar{X}(\textbf{1}\otimes I_s)
-2(\textbf{1}'\otimes I_s)X(\textbf{1}\otimes I_s))]\\
&=&\frac{8}{n}\left[\sum_{i=1}^nX_{ii}-I_s\right].
\end{eqnarray*}
Then from (\ref{taurtau}) we get
\begin{eqnarray} \label{taurtaufinal}\tau'R\tau&=&2\bar{x}'L\bar{x}+\frac{8}{n}\left[\sum_{i=1}^nX_{ii}-I_s\right].
\end{eqnarray}
For $i=1,2,\cdots,n$, the matrices $X_{ii}-\frac{1}{n}I_s$ are the
diagonal blocks of the matrix $L^\dagger $. Since
$\sum_{i=1}^nX_{ii}-I_s=\sum_{i=1}^n(X_{ii}-\frac{1}{n}I_s)$,  by
applying Lemma \ref{nonsingprin},  we get the matrices $X_{ii}-\frac{1}{n}I_s$
are positive definite  for $i=1,2,\cdots,n$. Again,
$\bar{x}'L\bar{x}=\bar{x}'QQ'\bar{x}$ is also positive semi definite
matrix. Hence, from (\ref{taurtaufinal}), we get $\tau'R\tau$ is a
positive definite matrix.
\end{proof}
\section{Determinant, inverse and inertia of the resistance matrix}\label{det-inv-iner}

In this section, we derive formulae for  determinant, inverse and inertia of the resistance matrix.

Consider the partition of the Laplacian matrix $L$ as in (\ref{Apartition})
\begin{eqnarray} L=\left[
  \begin{array}{cccc}
L_{11} & L_{12} & \cdots & L_{1n}\\
L_{21} & L_{22} & \cdots & L_{2n}\\
\cdots & \cdots &\cdots &\cdots\\
L_{n1} & L_{n2} & \cdots & L_{nn}\\
\end{array} \right].
\end{eqnarray}
Then, by definition, we have $\sum_{j=1}^nL_{ij}=0$ and
$\sum_{i=1}^nL_{ij}=0$ for all $i,j\in\{1,2,\cdots,n\}$. By applying
Lemma \ref{equalcofactor}, we get the cofactors of any two blocks $L_{ij}$
and $L_{kl}$ in $L$ are equal. Using this fact, first we derive a formula for the determinant of the resistance matrix.
\begin{theo}\label{rdeterminant}
Let $G$ be a connected graph on $n$ vertices such that the edge weights
are $s\times s$ positive definite matrices, and let $L$ and $R$ be the Laplacian matrix  and the resistance matrix of $G$, respectively. Then $\det R=(-1)^{(n-1)s}2^{(n-3)s}\frac{\det(\tau'R\tau)}{\chi(G)}$,
where $\chi(G)$ is the cofactor of any block of $L$.
\end{theo}
\begin{proof}By  Theorem \ref{taurtaunonsing}, we have $\tau'R\tau$ is a
nonsingular matrix. Using the Schur  formula for the determinant, we have
\begin{eqnarray}\label{Lpart}\nonumber\det\left[
                                \begin{array}{cc}
                                  -\frac{1}{2}L & \tau \\
                                  \tau' & -\tau'R\tau \\
                                \end{array}
                              \right]
\nonumber&=&\det(-\tau'R\tau)\det(-\frac{1}{2}L+\tau(\tau'R\tau)^{-1}\tau')\\
\nonumber&=&\det(-\tau'R\tau)\det R^{-1}\\
\label{detr}&=&(-1)^s\det(\tau'R\tau)\det R^{-1}.
\end{eqnarray}
Using Remark \ref{1tau1}, adding all the column blocks except the last column block to the first column block and  adding all the row blocks
except the last row block to the first row block of the matrix $\left[
                                \begin{array}{cc}
                                  -\frac{1}{2}L & \tau \\
                                  \tau' & -\tau'R\tau \\
                                \end{array}
                              \right]$ , we get
\begin{eqnarray*}\det\left[
                                \begin{array}{c|c}
                                  -\frac{1}{2}L & \tau \\\hline
                                  \tau' & -\tau'R\tau \\
                                \end{array}
                              \right]
&=&\det\left[
                                \begin{array}{c|c|c}
                                \vspace{.2cm} \textbf{0}_{s\times s} & \textbf{0}_{s\times (n-1)s} & 2I_s \\ \hline
                                \vspace{.2cm} \textbf{0}_{(n-1)s\times s} & -\frac{1}{2}L(1,1) & * \\\hline
                                \vspace{.2cm} 2I_s & * & -\tau'R\tau \\
                                \end{array}
                              \right],\\
&&~~~~~~~~~~~~~~~~~~~~\mbox{where $L(1,1)$ is the cofactor of $L_{1,1}$ in $L$.}\\
&=&(-1)^{(1+2+\cdots+s)+(ns+1+ns+2+\cdots+ns+s)}\det(2I_s)\det\left[
                                \begin{array}{c|c}
                                \vspace{.2cm} \textbf{0}_{(n-1)s\times s} & -\frac{1}{2}L(1,1) \\\hline
                                \vspace{.2cm} 2I_s & * \\
                                \end{array}
                              \right],~\\
&&~~~~~~~~~~~~~~~~~~~~~~~~~~~~~~~~~~~~~~~~~~~~~~~~~~~~\mbox{by expanding Laplace expansion.}\\
&=&(-1)^{ns^2}2^s\det\left[
                                \begin{array}{c|c}
                                \vspace{.2cm} \textbf{0}_{(n-1)s\times s} & -\frac{1}{2}L(1,1) \\\hline
                                \vspace{.2cm} 2I_s & * \\
                                \end{array}
                              \right]\\
&=&(-1)^{ns^2}2^s(-1)^{(n-1)s^2}\det(2I_s)\det(-\frac{1}{2}L(1,1))\\
&=&\frac{(-1)^{ns}}{2^{(n-3)s}}\det L(1,1)\\
&=&\frac{(-1)^{ns}}{2^{(n-3)s}}\chi(G).
\end{eqnarray*}
Using this in (\ref{Lpart}) we get
$$\det R=(-1)^{(n-1)s}2^{(n-3)s}\frac{\det(\tau'R\tau)}{\chi(G)}.$$
\end{proof}
In the next theorem, we establish that the resistance matrices are nonsingular and  derive a formula for the inverse of them.
\begin{theo}\label{rinverse}
Let $G$ be a connected graph on $n$ vertices such that the edge weights
are $s\times s$ positive definite matrices, and let $L$ and $R$ be the Laplacian matrix  and the resistance matrix of $G$, respectively. Then  $R$ is nonsingular and $R^{-1}=-\frac{1}{2}L+\tau(\tau'R\tau)^{-1}\tau'$.
\end{theo}
\begin{proof}
From equation (\ref{rwideq_1}), we have
\begin{eqnarray}\nonumber LR&=&L\bar{X}(J\otimes I_s)-2I_{ns}+\frac{2}{n}J\otimes I_s\\
\nonumber\Rightarrow LR+2I_{ns}&=&[L\bar{X}(\emph{\textbf{1}}\otimes I_s)+\frac{2}{n}(\emph{\textbf{1}}\otimes I_s)](\emph{\textbf{1}}'\otimes I_s)\\
\label{lrtau1}&=&\tau(\emph{\textbf{1}}'\otimes I_s)~~\mbox{[By Lemma \ref{taudef}]}\\
\nonumber\Rightarrow (LR+2I_{ns})\tau &=&\tau(\emph{\textbf{1}}'\otimes I_s)\tau\\
\nonumber&=&2\tau ~~\mbox{[By Remark \ref{1tau1}]}\\
\label{lrtau0}\Rightarrow LR\tau&=&0.
\end{eqnarray}
From  Theorem \ref{taurtaunonsing}, we  have $\tau^'R\tau$ is a
positive definite matrix, so $R\tau$ is  nonzero. Since $L$ has exactly $s$ number of zero
eigenvalues and $L(\emph{\textbf{1}}\otimes I_s)=0$, the column space of $R\tau$ is  a subspace
of column space of $(\emph{\textbf{1}}\otimes I_s)$. Thus, there exist
a matrix $C$ such that
\begin{eqnarray}\nonumber R\tau &=&(\emph{\textbf{1}} \otimes I_s)C\\
\nonumber \Rightarrow \tau' R\tau &=&\tau'(\emph{\textbf{1}}\otimes I_s)C\\
\nonumber &=&2C ~~\mbox{[By Remark \ref{1tau1}]}\\
\nonumber \Rightarrow C&=&\frac{1}{2}\tau'R\tau.
\end{eqnarray}
Now,
\begin{eqnarray}\nonumber R\tau &=&\frac{1}{2}(\emph{\textbf{1}}\otimes I_s)\tau'R\tau\\
\label{taur} \Rightarrow
\tau'R&=&\frac{1}{2}\tau'R\tau(\emph{\textbf{1}}'\otimes I_s),
\end{eqnarray}
and
\begin{eqnarray}\nonumber (-\frac{1}{2}L+\tau(\tau'R\tau)^{-1}\tau')R &=&-\frac{1}{2}LR+\tau(\tau'R\tau)^{-1}\tau'R\\
\nonumber  &=&-\frac{1}{2}LR+\frac{1}{2}\tau(\tau'R\tau)^{-1}\tau'R\tau(\emph{\textbf{1}}'\otimes I_s)~~\mbox{[By (\ref{taur})]}\\
\nonumber &=&-\frac{1}{2}LR+\frac{1}{2}\tau(\emph{\textbf{1}}'\otimes I_s)\\
\nonumber &=&I_{ns}.~~\mbox{[By (\ref{lrtau1})]}
\end{eqnarray}
Thus, the matrix $R$ is nonsingular, and
$R^{-1}=-\frac{1}{2}L+\tau(\tau'R\tau)^{-1}\tau'$.
\end{proof}
\begin{re}\label{qrq}
\emph{  Using (\ref{lrtau1}) we get
\begin{eqnarray*}&&~~LR=\tau(\emph{\textbf{1}}'\otimes I_s)-2I_{ns}\\
&&\Rightarrow LRQ=\tau(\emph{\textbf{1}}'\otimes I_s)Q-2Q\\
&&\Rightarrow QQ'RQ=-2Q~~[\mbox{As $(\emph{\textbf{1}}'\otimes
I_s)Q=\textbf{0}$}]\\
&&\Rightarrow Q'RQ=-2I_{(n-1)s}.~~[\mbox{As $Q$ has full column
rank}]
\end{eqnarray*}}
\end{re}
Interlacing inequality for the eigenvalues of distance and Laplace
matrices of a matrix weighted tree is given in \cite{fou-raj-bap}.
Using the fact $Q'RQ=-2I_{(n-1)s}$ we can similarly prove the
following:
\begin{theo}\label{Rinterlace}
Let $G$ be a weighted graph on $n$ vertices, where each weight is a
positive definite matrix of order $s$. Let $R$ be the resistance
matrix of $G$ and $L$ denote the Laplacian matrix of $G$. Let
$\mu_1\geq\mu_2\geq\cdots\geq\mu_{ns}$ be the eigenvalues of $R$ and
$\lambda_1\geq\lambda_2\geq\cdots\geq\lambda_{ns-s}>\lambda_{ns-s+1}=\cdots=\lambda_{ns}=0$
be the eigenvalues of $L$. Then
$$\mu_{s+i}\leq-\frac{2}{\lambda_i}\leq\mu_i~~\mbox{for}~~i=1,2,\cdots,ns-s.$$
\end{theo}
For  an $n \times n$ symmetric matrix $A$, let $p_+(A), p_-(A)$ and $p_0(A)$ denote  the number of positive, negative and zero eigenvalues of A, respectively.  Then $3$-tuple
$(p_+(A), p_-(A), p_0(A))$ is called the inertia of the matrix $A$.

Next we derive the formula for the inertia of the resistance matrix of a graph with
matrix weights.
\begin{theo}\label{Rinertia}
Let $G$ be a weighted graph on $n$ vertices, where each weight is a
positive definite matrix of order $s$ and $R$ be the resistance
matrix of $G$. Then inertia of $R$ is $(s,ns-s,0)$.
\end{theo}
\begin{proof}
Let $\mu_1\geq\mu_2\geq\cdots\geq\mu_{ns}$ be the eigenvalues of
$R$. Using the previous theorem we have $\mu_{s+i}\leq0$, for all
$i=1,2,\cdots,(n-1)s$. Again in the matrix $R$, $0_{s\times s}$ is a
principal submatrix.  Since $R$ is nonsingular, using
interlacing theorem, we have $\mu_{i}\geq0$, for all
$i=1,2,\cdots,s$. Hence the inertia of $R$ is $(s,ns-s,0).$
\end{proof}
\textbf{Acknowledgement:}The second author acknowledges the support
of the JC Bose Fellowship, Department of Science Technology,
Government of India. M. Rajesh Kannan would like to thank Department
of Science and Technology for the financial support.
\bibliographystyle{amsplain}
\bibliography{arxiv_resis}
%
%
%
%
%
%
%
%
%
%
%
%
%
%

\end{document}